\documentclass{birkjour}
\usepackage{amsmath, amsthm, amsfonts, amssymb}
\usepackage{mathrsfs}
\usepackage{txfonts}

\newtheorem{theorem}{Theorem}[section]
\newtheorem{definition}{Definition}[section]

\newtheorem{lemma}{Lemma}[section]

\theoremstyle{definition}
\newtheorem{remark}{Remark}[section]

\allowdisplaybreaks

\usepackage[numbers]{natbib}
\setcitestyle{open={},close={}}

\def \l {\left}
\def \r {\right}

\def \bR {\Bbb R}

\def\be{\begin{equation}}
\def\ee{\end{equation}}
\def\ba{\begin{array}}
	\def\ea{\end{array}}
\def\bd{\begin{definition}}
	\def\ed{\end{definition}}
\def\bt{\begin{theorem}}
	\def\et{\end{theorem}}
\def\bc{\begin{corollary}}
	\def\ec{\end{corollary}}
\def\bl{\begin{lemma}}
	\def\el{\end{lemma}}
\def\bdm{\begin{displaymath}}
\def\edm{\end{displaymath}}

\begin{document}

%
%
%
%
%
%
%
%
%

\title[Hardy-type operators on mixed radial-angular spaces]
 {Sharp bounds for Hardy-type operators on mixed radial-angular spaces}

\author[M. Wei]{Mingquan Wei$^*$\footnote {$^*$ Corresponding author}}

\address{%
School of Mathematics and Stastics, Xinyang Normal University\\
Xinyang 464000, China}

\email{weimingquan11@mails.ucas.ac.cn}

\author[D. Yan]{Dunyan Yan}

\address{%
	School of Mathematical Sciences, University of Chinese Academy of Sciences\\
	Beijing 100049, China}

\email{ydunyan@ucas.ac.cn}



\date{September 26, 2021}

\begin{abstract}
In this paper, by using the rotation method, we calculate that the sharp bound for $n$-dimensional Hardy operator $\mathcal{H}$ on mixed radial-angular spaces. Furthermore, we also obtain the sharp bound for $n$-dimensional fractional Hardy operator $\mathcal{H}_\beta$ from $L^p_{|x|}L_{\theta}^{\bar{p}}(\bR^n)$ to $L^q_{|x|}L_{\theta}^{\bar{q}}(\bR^n)$, where  $0<\beta<n$, $1<p,q,\bar{p},\bar{q}<\infty$ and $1/p-1/q=\beta/n$. By using duality, the corresponding results for the dual operators $\mathcal{H}^*$ and $\mathcal{H}^*_\beta$ are also established. In addition, the sharp weak-type estimate for $\mathcal{H}$  is also considered.
\end{abstract}

\maketitle
\section{Introduction}
Let $f$ be a non-negative integrable function on
$\bR$. The classical
Hardy operator and its dual operator are defined by
\begin{equation*}\label{formula1-1}
H(f)(x) := \frac{1}{x}\int_0^{x}f(t)dt,~~~
H^*(f)(x) := \int_x^{\infty}\frac{f(t)}{t}dt,
\end{equation*}
respectively, where $x\neq0$.

As we know, the classical Hardy operator was initially introduced by Hardy [\cite{hardy1920note}], who showed the following Hardy inequalities: 
$$\|H(f)\|_{L^p}\le \frac{p}{p-1}\|f\|_{L^p},~~\|H^*(f)\|_{L^p}\le p\|f\|_{L^p},$$
where the constants $\frac{p}{p-1},~p$ are best possible.

Later, Hardy-type operators were extended to higher dimension by Faris [\cite{faris1976weak}]. In 1995, Christ and Grafakos [\cite{christ1995best}] gave an equivalent version of $n$-dimensional Hardy operator 
\begin{eqnarray}\label{formula1-a}
{\mathcal{H}}(f)(x) := \frac{1}{\Omega_n|x|^n}\int_{|y|<|x|}f(y)dy,\,x\in \bR^n\backslash\{0\},
\end{eqnarray}
where $f$ is a non-negative measurable function on $\bR^n$ and $\Omega_n=\frac{\pi^{n/2}}{\Gamma(1+n/2)}$ is the volume of the unit ball in $\bR^n$. By a direct computation, the dual operator of ${\mathcal{H}}$ can be defined by setting, for any locally integrable function $f$ and $x\in\bR^n$, 
\begin{eqnarray}\label{formula1-a'}
{\mathcal{H}}^*(f)(x) :=\int_{|y|\geq|x|}\frac{f(y)}{\Omega_n|y|^n}dy,\,x\in \bR^n\backslash\{0\}.
\end{eqnarray}

By using the rotation method, Christ and Grafakos [\cite{christ1995best}] showed that the norms of $\mathcal{H}$ and ${\mathcal{H}}^*$ on $L^p(\bR^n)$ $(1<p<\infty)$ are also $\frac{p}{p-1}$ and $p$. Moreover, the sharp weak estimate for $\mathcal{H}$ was obtained by Zhao et al. [\cite{zhao2012endpoint}] as follows:

For $1\leq p\leq\infty$, we have
\begin{eqnarray}\label{formula1-6}
\|\mathcal{H}(f)\|_{L^{p,\infty}}\leq1\cdot\|f\|_{L^{p}},
\end{eqnarray} 
where the constant $1$ is best possible.

Similarly, for $0<\beta<n$, the $n$-dimensional fractional Hardy operator $\mathcal{H}_\beta$ and its dual operator $\mathcal{H}^*_\beta$ are defined by
\begin{eqnarray}\label{formula1-2}
{\mathcal{H}_\beta}(f)(x) := \frac{1}{(\Omega^{1/n}_n|x|)^{n-\beta}}\int_{|y|<|x|}f(y)dy,\,x\in \bR^n\backslash\{0\}
\end{eqnarray}
and
\begin{eqnarray}\label{formula1-2'}
{\mathcal{H}^*_\beta}(f)(x) :=\int_{|y|\geq|x|}\frac{f(y)}{(\Omega^{1/n}_n|y|)^{n-\beta}}dy,\,x\in \bR^n\backslash\{0\},
\end{eqnarray}
respectively.

In recent years, much attention has been paid to the sharp bounds for Hardy-type operators on different function spaces. For instance, the sharp bounds for Hardy-type operators and their dual operators on weak Lebesgue spaces were considered in [\cite{gao2016sharp,gao2015sharp,hussain2020sharp,mingquan2018sharp,sarfraz2021weak,yu2018sharp,zhao2012endpoint,zhao2015best}]. See also [\cite{fu2015weighted,fu2012sharp,lu2013sharp,wang2012explicit}] for sharp constants for multilinear Hardy operator on Lebesgue spaces and Morrey-type spaces. Moreover, Hardy-type inequalities have been extended to different settings. For example, Wu et al. [\cite{wu2016sharp}] and Guo et al. [\cite{guo2015hausdorffa}]  considered the sharp constants for Hardy operators and their dual operators on Heisenberg group in the linear and multilinear situations. Fu et al. [\cite{fu2013sharp,wu2017weighted}] investigated sharp constants for Hardy-type operators on $p$-adic field. Maligranda et al. [\cite{maligranda2014hardy}], Guo and Zhao [\cite{guo2017some}], Fan and Zhao [\cite{fan2019sharp}] studied Hardy $q$-type integral inequalities.  In addition, there are many  extensions of the Hardy-type operators, such as weighted Hardy-Littlewood averages [\cite{fu2009commutators,fu2010weighted,xiao2001lp}] and Hausdorff operators [\cite{chen2018boundedness,chuong2019two,liflyand2009boundedness,wu2017hardy}].
Readers can refer to the book [\cite{hardy1952littlewood}] to get some earlier development of Hardy-type inequalities. We also refer the readers to the review papers [\cite{kufner2006prehistory,lu2012ndim}] for some recent progress on Hardy-type operators and their related topics.

Recently the mixed radial-angular spaces have been successfully used in studying
Strichartz estimates and partial differential equations to improve the corresponding
results (see [\cite{cacciafesta2013endpoint,d2016regularity,fang2011weighted,luca2014regularity,tao2000spherically}], etc.). After that, many operators in harmonic analysis have been proved to be bounded on these spaces. For instance, the extrapolation theorems on mixed radial-angular spaces were build by Duoandikoetxea and Oruetxebarria [\cite{duoandikoetxea2019weighted}] to study the boundedness of a large class of operators which are weighted bounded. In addition, the boundedness of some operators with rough kernels on mixed radial-angular spaces were also considered by Liu et al. [\cite{liu2019weighted,liu2020mixed,liu2021mixed}].  Inspired by the references mentioned above, it is natural to ask whether we can obtain the boundedness, or furthermore the sharp bounds, for Hardy-type operators on mixed radial-angular spaces. In this paper, we will give an affirmative answer. More precisely, we will study the sharp bounds for $n$-dimensional Hardy operator $\mathcal{H}$ and $n$-dimensional fractional Hardy operator $\mathcal{H}_\beta$ on mixed radial-angular spaces. By using a duality argument, we also obtain the sharp constants for their dual operators $\mathcal{H}^*$ and $\mathcal{H}^*_\beta$ on mixed radial-angular spaces. Moreover, the sharp weak type estimate for $\mathcal{H}$ is also considered.

Now we recall the definition of mixed radial-angular spaces.
\begin{definition}\label{D-1-1}
	For $n\geq2$, $1\leq p,\bar{p}\leq\infty$, the mixed radial-angular space $L^p_{|x|}L_{\theta}^{\bar{p}}(\bR^n)$ consists of all functions $f$ in $\bR^n$ for which
	\begin{equation*}
	\|f\|_{L^p_{|x|}L_{\theta}^{\bar{p}}}:=\l(\int_{0}^\infty\l(\int_{\mathbb{S}^{n-1}}|f(r,\theta)|^{\bar{p}}d\theta\r)^{p/\bar{p}}r^{n-1}dr\r)^{1/p}<\infty,
	\end{equation*}
	where $\mathbb{S}^{n-1}$ denotes the unit sphere in $\bR^n$. If $p=\infty$ or $\bar{p}=\infty$, then we have to make appropriate modifications.
\end{definition}
Similarly, we can define the weak mixed radial-angular spaces.
\begin{definition}\label{D-1-2}
	For $n\geq2$, $1\leq p,\bar{p}\leq\infty$, the weak mixed radial-angular space $wL^p_{|x|}L_{\theta}^{\bar{p}}(\bR^n)$ consists of all functions $f$ in $\bR^n$ for which
	\begin{equation*}
	\|f\|_{wL^p_{|x|}L_{\theta}^{\bar{p}}}:=\sup_{\lambda>0}\lambda~\|\chi_{\{x\in\bR^n:|f(x)|>\lambda\}}\|_{L^p_{|x|}L_{\theta}^{\bar{p}}}<\infty,
	\end{equation*}
	where $\chi_{\{x\in\bR^n:|f(x)|>\lambda\}}$ denotes the characteristic function of the set $\{x\in\bR^n:|f(x)|>\lambda\}$.
\end{definition}
In fact, the mixed radial-angular spaces can be seen as particular cases of mixed-norm Lebesgue spaces introduced by  Benedek and Panzone [\cite{1961The}]. We refer readers to   [\cite{dong2019fully,huang2019atomic,huang2019function,nogayama2019boundedness,nogayama2020atomic,phan2019well}]
for more studies on mixed-norm Lebesgue spaces and their applications in PDE.

The organization of this article is as follows. The sharp bounds for $n$-dimensional Hardy operator and its dual operator on mixed radial-angular spaces are obtained in Sect. 2. We calculate the operator norms of  $n$-dimensional fractional Hardy operator on mixed radial-angular spaces in Sect. 3. In addition, we also establish a sharp weak type estimate for $n$-dimensional Hardy operator on mixed radial-angular spaces Sect. 4.

\section{Sharp bounds for  $\mathcal{H}$ and $\mathcal{H}^*$ from $L^p_{|x|}L_{\theta}^{\bar{p}_1}(\bR^n)$ to $L^p_{|x|}L_{\theta}^{\bar{p}_2}(\bR^n)$}

The main result in this section is the following:
\begin{theorem}\label{T-1}
	Let $n\geq2$, $1<p,\bar{p_1},\bar{p_2}<\infty$. Then the $n$-dimensional Hardy operator $\mathcal{H}$ defined in (\ref{formula1-a}) is bounded from $L^p_{|x|}L_{\theta}^{\bar{p}_1}(\bR^n)$ to $L^p_{|x|}L_{\theta}^{\bar{p}_2}(\bR^n)$. Moreover,
	\begin{equation*}
	\|\mathcal{H}\|_{L^p_{|x|}L_{\theta}^{\bar{p}_1}\rightarrow L^p_{|x|}L_{\theta}^{\bar{p}_2}}=\frac{p}{p-1}w_n^{1/\bar{p}_2-1/\bar{p}_1},
	\end{equation*}
	where $w_n=2\pi^{n/2}/\Gamma(n/2)$ is the induced measure of $\mathbb{S}^{n-1}$.
\end{theorem}
\begin{proof}
	We borrow some ideas from [\cite{christ1995best,fu2012sharp}] and use the method of rotation. Set 
	\begin{equation}\label{E-2-1}
	g(x)=\frac{1}{w_n}\int_{\mathbb{S}^{n-1}}f(|x|\theta)d\theta,\quad x\in\bR^n.
	\end{equation}
	
	Obviously, $g$ is a radial function and it was proved by Fu et al [\cite{fu2012sharp}, Proof of Theorem 1] that  $\mathcal{H}(g)$ is equal to $\mathcal{H}(f)$. Moreover, we have
	\begin{eqnarray}\label{E-2-4}
	\|g\|_{L^p_{|x|}L_{\theta}^{\bar{p}_1}}
	&=&\l(\int_{0}^\infty\l(\int_{\mathbb{S}^{n-1}}|g(r,\theta)|^{\bar{p}_1}d\theta\r)^{p/\bar{p}_1}r^{n-1}dr\r)^{1/p}\nonumber\\
	&=&w^{1/\bar{p}_1}_n\l(\int_{0}^\infty |g(r)|^{p}r^{n-1}dr\r)^{1/p},
	\end{eqnarray}
	where $g(r)$ can be recognized as
	$g(r)=g(x)$
	for any $x\in\bR^n$ with $|x|=r$, since $g$ is a radial function.
	
	Combining (\ref{E-2-1}) and (\ref{E-2-4}), and using H{\"o}lder's inequality, we get 
	\begin{eqnarray*}
		\|g\|_{L^p_{|x|}L_{\theta}^{\bar{p}_1}}
		&=&w^{1/\bar{p}_1}_n\l(\int_{0}^\infty \l|\frac{1}{w_n}\int_{\mathbb{S}^{n-1}}f(r\theta)d\theta\r|^{p}r^{n-1}dr\r)^{1/p}\\
		&=&w^{1/\bar{p}_1-1}_n\l(\int_{0}^\infty \l|\int_{\mathbb{S}^{n-1}}f(r\theta)d\theta\r|^{p}r^{n-1}dr\r)^{1/p}\\
		&\leq&w^{1/\bar{p}_1-1}_n\l(\int_{0}^\infty \l(\int_{\mathbb{S}^{n-1}}|f(r\theta)|^{\bar{p}_1}d\theta\r)^{p/\bar{p}_1}\l(\int_{\mathbb{S}^{n-1}}d\theta\r)^{p/\bar{p}'_1}r^{n-1}dr\r)^{1/p}\\
		&=&\l(\int_{0}^\infty \l(\int_{\mathbb{S}^{n-1}}|f(r\theta)|^{\bar{p}_1}d\theta\r)^{p/\bar{p}_1}r^{n-1}dr\r)^{1/p}= \|f\|_{L^p_{|x|}L_{\theta}^{\bar{p}_1}}.
	\end{eqnarray*}
	Therefore we have that 
	\begin{eqnarray*}
		\frac{\|\mathcal{H}(f)\|_{L^p_{|x|}L_{\theta}^{\bar{p}_2}}}{\|f\|_{L^p_{|x|}L_{\theta}^{\bar{p}_1}}}\leq \frac{\|\mathcal{H}(g)\|_{L^p_{|x|}L_{\theta}^{\bar{p}_2}}}{\|g\|_{L^p_{|x|}L_{\theta}^{\bar{p}_1}}}.
	\end{eqnarray*}
	That is to say, the operator $\mathcal{H}$ and its restriction to radial functions have the same
	operator norm from $L^p_{|x|}L_{\theta}^{\bar{p}_1}(\bR^n)$ to $L^p_{|x|}L_{\theta}^{\bar{p}_2}(\bR^n)$. Therefore, we can assume that $f$ is a radial function. 
	
	For a radial function $f$, $\mathcal{H}(f)$ is also a radial function. Consequently,  
	\begin{eqnarray}\label{E-2-2}
	\|\mathcal{H}(f)\|_{L^p_{|x|}L_{\theta}^{\bar{p}_2}}&=&\l(\int_{0}^\infty\l(\int_{\mathbb{S}^{n-1}}|\mathcal{H}(f)(r,\theta)|^{\bar{p}_2}d\theta\r)^{p/\bar{p}_2}r^{n-1}dr\r)^{1/p}\nonumber\\
	&=&\l(\int_{0}^\infty\l(\int_{\mathbb{S}^{n-1}}|\mathcal{H}(f)(r)|^{\bar{p}_2}d\theta\r)^{p/\bar{p}_2}r^{n-1}dr\r)^{1/p}\nonumber\\
	&=&w^{1/\bar{p}_2}_n\l(\int_{0}^\infty|\mathcal{H}(f)(r)|^{p}r^{n-1}dr\r)^{1/p},
	\end{eqnarray}
	where $\mathcal{H}(f)(r)$ can be recognized as
	$\mathcal{H}(f)(r)=\mathcal{H}(f)(x)$
	for any $x\in\bR^n$ with $|x|=r$, since $\mathcal{H}(f)$ is a radial function.
	
	Denote $B(0, R)$ by the ball centered at the origin with radius $R$. By changing variables, we further have
	\begin{equation}\label{E-2-3}
	\mathcal{H}(f)(r)=\frac{1}{\Omega_n}\int_{B(0,1)}f(ry)dy.
	\end{equation}
	Combining (\ref{E-2-2}) with (\ref{E-2-3}), and usiing Minkowski's inequality, we have
	\begin{eqnarray*}
		\|\mathcal{H}(f)\|_{L^p_{|x|}L_{\theta}^{\bar{p}_2}}&=&\frac{w_n^{1/\bar{p}_2}}{\Omega_n}\l(\int_{0}^\infty\l|\int_{B(0,1)}f(ry)dy\r|^{p}r^{n-1}dr\r)^{1/p}\nonumber\\
		&\leq&\frac{w_n^{1/\bar{p}_2}}{\Omega_n}
		\int_{B(0,1)}\l(\int_{0}^\infty|f(ry)|^pr^{n-1}dr\r)^{1/p} dy\nonumber\\
		&=&\frac{w_n^{1/\bar{p}_2}}{\Omega_n}
		\int_{B(0,1)}\l(\int_{0}^\infty|f(r|y|)|^pr^{n-1}dr\r)^{1/p} dy\nonumber\\
		&=&\frac{w_n^{1/\bar{p}_2-1/\bar{p}_1}}{\Omega_n}
		\int_{B(0,1)}\l(\int_{0}^\infty w_n^{p/\bar{p}_1} |f(r)|^pr^{n-1}dr\r)^{1/p} |y|^{-n/p}dy\nonumber\\
		&=&\frac{w_n^{1/\bar{p}_2-1/\bar{p}_1}}{\Omega_n}
		\int_{B(0,1)} |y|^{-n/p}dy \|f\|_{L^p_{|x|}L_{\theta}^{\bar{p}_1}}\nonumber\\
		&=&w_n^{1/\bar{p}_2-1/\bar{p}_1}\frac{w_n}{\Omega_n}
		\frac{p}{n(p-1)} \|f\|_{L^p_{|x|}L_{\theta}^{\bar{p}_1}}\\
		&=&\frac{p}{p-1}w_n^{1/\bar{p}_2-1/\bar{p}_1}\|f\|_{L^p_{|x|}L_{\theta}^{\bar{p}_1}},
	\end{eqnarray*}
	where we have used the identity $w_n=n\Omega_n$.
	
	To prove the constant $\frac{p}{p-1}w_n^{1/\bar{p}_2-1/\bar{p}_1}$ is best possible, we take
	\begin{eqnarray*}
		f_{\epsilon}(x)=\left\{\begin{array}{ll}
			0, & |x|\leq 1,\\
			|x|^{-(\frac{n}{p}+\epsilon)},& |x|> 1, 
		\end{array}\right.
	\end{eqnarray*}
	where $0<\epsilon<1$.
	A direct computation yields 
	\begin{eqnarray*}
		\|f_\epsilon\|_{L^p_{|x|}L_{\theta}^{\bar{p}_1}}=\frac{w^{1/\bar{p}_1}_n}{(p\epsilon)^{1/p}}.
	\end{eqnarray*}
	On the other hand, 
	\begin{eqnarray*}
		\mathcal{H}(f_{\epsilon})(x)=\left\{\begin{array}{ll}
			0, & |x|\leq 1,\\
			\frac{1}{\Omega_n}|x|^{-(\frac{n}{p}+\epsilon)}\int_{\frac{1}{|x|}<|y|<1}|y|^{-(\frac{n}{p}+\epsilon)}dy,& |x|> 1. 
		\end{array}\right.
	\end{eqnarray*}
	So we have
	\begin{eqnarray*}
		\|\mathcal{H}(f_\epsilon)\|_{L^p_{|x|}L_{\theta}^{\bar{p}_2}}&=&\frac{w_n^{1/\bar{p}_2}}{\Omega_n}\l(\int_{1}^\infty\l|r^{-(\frac{n}{p}+\epsilon)}\int_{\frac{1}{r}<|y|<1}|y|^{-(\frac{n}{p}+\epsilon)}dy\r|^{p}r^{n-1}dr\r)^{1/p}\nonumber\\
		&\geq&\frac{w_n^{1/\bar{p}_2}}{\Omega_n}\l(\int_{1/{\epsilon}}^\infty\l|r^{-(\frac{n}{p}+\epsilon)}\int_{\epsilon<|y|<1}|y|^{-(\frac{n}{p}+\epsilon)}dy\r|^{p}r^{n-1}dr\r)^{1/p}\nonumber\\
		&=&\frac{w_n^{1/\bar{p}_2}}{\Omega_n}\l(\int_{1/\epsilon}^\infty r^{-p\epsilon-1}dr\r)^{1/p}
		\int_{\epsilon<|y|<1}|y|^{-(\frac{n}{p}+\epsilon)}dy\nonumber\\
		&=&\frac{w_n^{1+1/\bar{p}_2}}{\Omega_n}\l(\int_{1/\epsilon}^\infty r^{-p\epsilon-1}dr\r)^{1/p}
		\int_{\epsilon}^1r^{n-1-(\frac{n}{p}+\epsilon)}dr\nonumber\\
		&=&\frac{w_n^{1+1/\bar{p}_2}}{\Omega_n}\frac{\epsilon^\epsilon}{(p\epsilon)^{1/p}}\frac{1-\epsilon^{n-\epsilon-n/p}}{n-\epsilon-n/p}\nonumber\\
		&=&n\epsilon^\epsilon\frac{1-\epsilon^{n-\epsilon-n/p}}{n-\epsilon-n/p}\times w_n^{1/\bar{p}_2-1/\bar{p}_1}\times\|f_\epsilon\|_{L^p_{|x|}L_{\theta}^{\bar{p}_1}}
	\end{eqnarray*}
	As a consequence, we obtain 
	$$\|\mathcal{H}\|_{L^p_{|x|}L_{\theta}^{\bar{p}_1}\rightarrow L^p_{|x|}L_{\theta}^{\bar{p}_2}}\geq n\epsilon^\epsilon\frac{1-\epsilon^{n-\epsilon-n/p}}{n-\epsilon-n/p}\times w_n^{1/\bar{p}_2-1/\bar{p}_1}.$$
	Letting $\epsilon\rightarrow0^+$, and using  $\epsilon^\epsilon\rightarrow1$, it yields
	$$\|\mathcal{H}\|_{L^p_{|x|}L_{\theta}^{\bar{p}_1}\rightarrow L^p_{|x|}L_{\theta}^{\bar{p}_2}}\geq \frac{p}{p-1}w_n^{1/\bar{p}_2-1/\bar{p}_1},$$
	which finishes the proof.
\end{proof}
\begin{remark}\label{R-2-1}
	When $p=\bar{p}_1=\bar{p}_2\in(1,\infty)$ in Theorem \ref{T-1}, 
	 we recover the results in [\cite{christ1995best,fu2012sharp}].
\end{remark}
A standard argument yields that for $1<p,\bar{p}<\infty$, the dual space of $L^p_{|x|}L_{\theta}^{\bar{p}}(\bR^n)$ is $L^{p'}_{|x|}L_{\theta}^{\bar{p}'}(\bR^n)$, where $p'$ and $\bar{p}'$ satisfy $1/p+1/p'=1$ and $1/\bar{p}+1/\bar{p}'=1$ (see [\cite{1961The}]), and furthermore, for any $f\in L^{p'}_{|x|}L_{\theta}^{\bar{p}'}(\bR^n)$,
$$\|f\|_{L^{p'}_{|x|}L_{\theta}^{\bar{p}'}}=\sup_{\|g\|_{L^{p}_{|x|}L_{\theta}^{\bar{p}}}\leq1}
\int_{\bR^n}f(x)g(x)dx.$$
 This observation enables us to obtain the sharp bound for $\mathcal{H}^*$ from $L^p_{|x|}L_{\theta}^{\bar{p}_1}(\bR^n)$ to $L^p_{|x|}L_{\theta}^{\bar{p}_2}(\bR^n)$ by using  duality. 
\begin{theorem}\label{T-2}
	Let $n\geq2$, $1<p,\bar{p}_1,\bar{p}_2<\infty$. Then the operator $\mathcal{H}^*$ defined in (\ref{formula1-a'}) is bounded from $L^p_{|x|}L_{\theta}^{\bar{p}_1}(\bR^n)$ to $L^p_{|x|}L_{\theta}^{\bar{p}_2}(\bR^n)$. Moreover,
	\begin{equation*}
	\|\mathcal{H}^*\|_{L^p_{|x|}L_{\theta}^{\bar{p}_1}\rightarrow L^p_{|x|}L_{\theta}^{\bar{p}_2}}=pw_n^{1/\bar{p}_2-1/\bar{p}_1}.
	\end{equation*}
\end{theorem}

\section{Sharp bound for $\mathcal{H}_\beta$ from $L^p_{|x|}L_{\theta}^{\bar{p}}(\bR^n)$ to $L^q_{|x|}L_{\theta}^{\bar{q}}(\bR^n)$}
To state the main result in this section, we need the following Lemma.
\begin{lemma}\label{L-3-1}
	Let $n\geq2$, $0<\beta<n$ and $1<p<q<\infty$ such that $1/p-1/q=\beta/n$. Then the $n$-dimensional fractional Hardy operator $\mathcal{H}_\beta$ defined in (\ref{formula1-2}) is bounded from $L^p(\bR^n)$ to $L^q(\bR^n)$. Moreover,
	\begin{equation*}
	\|\mathcal{H}_\beta\|_{L^p\rightarrow L^q}=C_{p,q,n,\beta},
	\end{equation*}
	where $C_{p,q,n,\beta}=\l(\frac{p'}{q}\r)^{1/q}\l(\frac{n}{q\beta}\cdot B\l(\frac{n}{q\beta},\frac{n}{q'\beta}\r)\r)^{-\beta/n}$ and $B$ is the Beta function, i.e., $B(z_1,z_2)=\int_{0}^1t^{z_1-1}(1-t)^{z_2-1}dt$ for any  complex numbers $z_1$ and $z_1$ with the positive real parts.
\end{lemma}
Our main result in this section can be stated as follws:
\begin{theorem}\label{T-3}
	Let $n\geq2$, $0<\beta<n$, $1<\bar{p},\bar{q}<\infty$ and $1<p<q<\infty$ such that $1/p-1/q=\beta/n$. Then the $n$-dimensional fractional Hardy operator $\mathcal{H}_\beta$ defined in (\ref{formula1-2}) is bounded from $L^p_{|x|}L_{\theta}^{\bar{p}}(\bR^n)$ to $L^q_{|x|}L_{\theta}^{\bar{q}}(\bR^n)$. Moreover,
	\begin{equation*}
	\|\mathcal{H}_\beta\|_{L^p_{|x|}L_{\theta}^{\bar{p}}\rightarrow L^q_{|x|}L_{\theta}^{\bar{q}}}=C_{p,q,n,\beta}w_n^{1/\bar{q}-1/\bar{p}+\beta/n}.
	\end{equation*}
\end{theorem}
\begin{proof}
	A similar process as in the proof of Theorem \ref{T-1} yields that the norm of the operator $\mathcal{H}_\beta$ from $L^p_{|x|}L_{\theta}^{\bar{p}}(\bR^n)$ to $L^q_{|x|}L_{\theta}^{\bar{q}}(\bR^n)$ is equal to the norm of $\mathcal{H}_\beta$
	acting on radial functions.
	
	For a radial function $f$, $\mathcal{H}_\beta(f)$ is also a radial function. Consequently,  
	\begin{eqnarray*}\label{E-2-2'}
	\|\mathcal{H}_\beta(f)\|_{L^q_{|x|}L_{\theta}^{\bar{q}}}&=&\l(\int_{0}^\infty\l(\int_{\mathbb{S}^{n-1}}|\mathcal{H}_\beta(f)(r,\theta)|^{\bar{q}}d\theta\r)^{q/\bar{q}}r^{n-1}dr\r)^{1/q}\nonumber\\
	&=&w^{1/\bar{q}}_n\l(\int_{0}^\infty|\mathcal{H}_\beta(f)(r,\theta)|^{q}r^{n-1}dr\r)^{1/q}\nonumber\\
	&=&w^{1/\bar{q}-1/q}_n\l(\int_{0}^\infty\int_{\mathbb{S}^{n-1}}|\mathcal{H}_\beta(f)(r,\theta)|^{q}r^{n-1}d\theta dr\r)^{1/q}\nonumber\\
	&=&w^{1/\bar{q}-1/q}_n\|\mathcal{H}_\beta(f)\|_{L^q}\nonumber\\
	&\leq&w^{1/\bar{q}-1/q}_n\times C_{p,q,n,\beta}\cdot\|f\|_{L^p}\nonumber\\
	&=&C_{p,q,n,\beta}\cdot w^{1/\bar{q}-1/q}_n\cdot w^{1/p-1/\bar{p}}_n \cdot\|f\|_{L^p_{|x|}L_{\theta}^{\bar{p}}}\nonumber\\
	&=&C_{p,q,n,\beta}w_n^{1/\bar{q}-1/\bar{p}+\beta/n} \|f\|_{L^p_{|x|}L_{\theta}^{\bar{p}}},
	\end{eqnarray*}
	where we have used Lemma \ref{L-3-1} in the inequality and $1/p-1/q=\beta/n$ in the last equality.
	
	To get the sharp bound, we take 
	$$f_0(x)=\frac{1}{(1+|x|^{q\beta})^{1+\frac{n}{q\beta}}}.$$
	
	Since $f_0$ is a radial function, we have
	\begin{eqnarray}\label{E-3-2}
	\|f_0\|_{L^p_{|x|}L_{\theta}^{\bar{p}}}=w^{1/\bar{p}-1/{p}}_n\|f_0\|_{L^p}.
	\end{eqnarray}
	Similarly, noting that $\mathcal{H}_\beta(f_0)$ is also a radial function, we have
	\begin{eqnarray}\label{E-3-3}
	\|\mathcal{H}_\beta(f_0)\|_{L^q_{|x|}L_{\theta}^{\bar{q}}}=w^{1/\bar{q}-1/{q}}_n\|\mathcal{H}_\beta(f_0)\|_{L^q}.
	\end{eqnarray}
	Due to Zhao and Lu [\cite{zhao2015best}], there holds
	\begin{eqnarray}\label{E-3-4}
	\|\mathcal{H}_\beta(f_0)\|_{L^q}=C_{p,q,n,\beta}\|f_0\|_{L^p}.
	\end{eqnarray}
	Combining (\ref{E-3-2}), (\ref{E-3-3}) with (\ref{E-3-4}), we arrive at 
	\begin{eqnarray*}
		\|\mathcal{H}_\beta(f_0)\|_{L^q_{|x|}L_{\theta}^{\bar{q}}}=C_{p,q,n,\beta}w_n^{1/\bar{q}-1/\bar{p}+\beta/n}\|f_0\|_{L^p_{|x|}L_{\theta}^{\bar{p}}},
	\end{eqnarray*}
	since $1/p-1/q=\beta/n$.
	
	The proof is finished.
\end{proof}
\begin{remark}
	Similar to Theorem \ref{T-2}, one can also calculate the sharp bound for $\mathcal{H}^*_\beta$ defined in (\ref{formula1-2'}) from $L^p_{|x|}L_{\theta}^{\bar{p}}(\bR^n)$ to $L^q_{|x|}L_{\theta}^{\bar{q}}(\bR^n)$ by duality. We omit the details here. In particular, if $p=\bar{p}\in(1,\infty)$ and $q=\bar{q}\in(1,\infty)$ in Theorem \ref{T-3}, we recover the results in [\cite{persson2015note,zhao2015best}].
\end{remark}

\section{Sharp weak bound for $\mathcal{H}$ from $L^p_{|x|}L_{\theta}^{\bar{p}_1}(\bR^n)$ to $wL^p_{|x|}L_{\theta}^{\bar{p}_2}(\bR^n)$}
This section considers the sharp weak type estimate for $n$-dimensional Hardy operator on mixed radial-angular spaces. Our main result can be read as follows.
\begin{theorem}\label{T-4-1}
	Let $n\geq2$, $1\leq p,\bar{p}_1,\bar{p}_2\leq\infty$. Then the $n$-dimensional Hardy operator $\mathcal{H}$ defined in (\ref{formula1-a}) is bounded from $L^p_{|x|}L_{\theta}^{\bar{p}_1}(\bR^n)$ to $wL^p_{|x|}L_{\theta}^{\bar{p}_2}(\bR^n)$. Moreover,
	\begin{equation*}
	\|\mathcal{H}\|_{L^p_{|x|}L_{\theta}^{\bar{p}_1}\rightarrow wL^p_{|x|}L_{\theta}^{\bar{p}_2}}=w_n^{1/\bar{p}_2-1/\bar{p}_1}.
	\end{equation*}
\end{theorem}
\begin{proof}
	We only give the proof of the case $1<p,\bar{p}_1,\bar{p}_2<\infty$, with the usual modifications made
	when $p=1,\bar{p}_i=1$ or $p=\infty,\bar{p}_i=\infty$, $i=1,2$. For any $\lambda>0$, we have
	\begin{eqnarray*}
		\|\chi_{\{x\in\bR^n: |\mathcal{H}(f)(x)|>\lambda\}}\|_{L^{p}_{|x|}L_{\theta}^{\bar{p}_2}}
		&=&\l\|\chi_{\l\{x\in\bR^n: \l|\frac{1}{\Omega_n|x|^n}\int_{|y|<|x|}f(y)dy\r|>\lambda\r\}}\r\|_{L^{p}_{|x|}L_{\theta}^{\bar{p}_2}}\\
		&\leq&\l\|\chi_{\l\{x\in\bR^n: \frac{1}{\Omega_n|x|^n}\|f\chi_{\{|y|<|x|\}}\|_{L^{p}_{|x|}L_{\theta}^{\bar{p}_1}}
			\|\chi_{\{|y|<|x|\}}\|_{L^{p'}_{|x|}L_{\theta}^{\bar{p}'_1}}>\lambda\r\}}\r\|_{L^{p}_{|x|}L_{\theta}^{\bar{p}_2}}\\
		&\leq&\l\|\chi_{\l\{x\in\bR^n: \frac{w_n^{1/\bar{p}'_1-1/p'}}{ \Omega_n^{1/p}|x|^{n/p} }\|f\chi_{\{|y|<|x|\}}\|_{L^{p}_{|x|}L_{\theta}^{\bar{p}_1}} >\lambda\r\}}\r\|_{L^{p}_{|x|}L_{\theta}^{\bar{p}_2}}\\
		&\leq&\l\|\chi_{\l\{x\in\bR^n: |x|^n<\frac{w_n^{p/\bar{p}'_1-p/p'}}{\lambda^p\Omega_n}\|f\|^p_{L^{p}_{|x|}L_{\theta}^{\bar{p}_1}}\r\}}\r\|_{L^{p}_{|x|}L_{\theta}^{\bar{p}_2}}\\
		&=&w_n^{1/\bar{p}_2}
		\l(\int_{0}^{\l(\frac{w_n^{p/\bar{p}'_1-p/p'}}{\lambda^p\Omega_n}\|f\|^p_{L^{p}_{|x|}L_{\theta}^{\bar{p}_1}}\r)^{1/n}}r^{n-1}dr\r)^{1/p}\\
		&=&\frac{w_n^{1/\bar{p}_2-1/\bar{p}_1}}{\lambda}\times \|f\|_{L^{p}_{|x|}L_{\theta}^{\bar{p}_1}},
	\end{eqnarray*}
	where we have used H{\"o}lder's inequality on mixed-norm Lebesgue spaces, see [\cite{1961The}].
	
	On the other hand, we need to show that the constant $w_n^{1/\bar{p}_2-1/\bar{p}_1}$ is the best possible.  Denote by
	$\chi_r=\chi_{[0,r]},\,r>0$. Taking $f_0(x)=\chi_r(|x|), x\in\bR^n$, a simple calculation shows that
	\begin{equation*}
	\|f_0\|_{L^{p}_{|x|}L_{\theta}^{\bar{p}_1}}=\frac{w_n^{1/\bar{p}_1}}{n^{1/p}}r^{n/p}.
	\end{equation*}
	Zhao et al. [\cite{zhao2012endpoint}, Proof of Theorem 2.1] proved that for $x\in\bR^n$, $\mathcal{H}(f)(x)\leq1$.
	Therefore, for $0<\lambda\leq1$, we divide $x\in\bR^n$ into two cases.
	
	(i) When $|x|<r$, it was showed by Zhao et al. [\cite{zhao2012endpoint}, Proof of Theorem 2.1] that $\mathcal{H}(f_0)(x)= 1$. As a consequence, we have 
	\begin{equation*}
	\|\chi_{\{|x|<r: |\mathcal{H}(f_0)(x)|>\lambda\}}\|^p_{L^{p}_{|x|}L_{\theta}^{\bar{p}_2}}= w_n^{p/\bar{p}_2}
	\frac{r^n}{n}.
	\end{equation*}
	
	(ii) When $|x|\geq r$, it was also showed by Zhao et al. [\cite{zhao2012endpoint}, Proof of Theorem 2.1] that $\mathcal{H}(f_0)(x)= r^n/|x|^n$. Therefore we have 
	\begin{eqnarray*}
		\|\chi_{\{|x|\geq r: |\mathcal{H}(f_0)(x)|>\lambda\}}\|^p_{L^{p}_{|x|}L_{\theta}^{\bar{p}_2}}= \l\|\chi_{\l\{x\in\bR^n: r\leq|x|< \frac{r}{\lambda^{1/n}}\r\}}\r\|^p_{L^{p}_{|x|}L_{\theta}^{\bar{p}_2}}
		= w_n^{p/\bar{p}_2}
		\frac{r^n}{n}\l(\frac{1}{\lambda}-1\r).\\
	\end{eqnarray*}
	From the above results, we have the following estimate,
	\begin{equation*}
	\|\chi_{\{x\in\bR^n: |\mathcal{H}(f_0)(x)|>\lambda\}}\|^p_{L^{p}_{|x|}L_{\theta}^{\bar{p}_2}}
	=\frac{r^{n}w_n^{p/\bar{p}_2}}{n\lambda}
	=\frac{w_n^{p/\bar{p}_2-p/\bar{p}_1}}{\lambda}\|f_0\|^p_{L^{p}_{|x|}L_{\theta}^{\bar{p}_1}},
	\end{equation*}
	which yields that for $1<p<\infty$, 
	\begin{eqnarray*}
		\|\mathcal{H}(f_0)\|_{wL^p_{|x|}L_{\theta}^{\bar{p}_2}}&=&\sup_{\lambda>0}\lambda\|\chi_{\{x\in\bR^n: |\mathcal{H}(f_0)(x)|>\lambda\}}\|_{L^{p}_{|x|}L_{\theta}^{\bar{p}_2}}\\
		&=&\sup_{0<\lambda\leq1}\lambda\|\chi_{\{x\in\bR^n: |\mathcal{H}(f_0)(x)|>\lambda\}}\|_{L^{p}_{|x|}L_{\theta}^{\bar{p}_2}}\\
		&=&w_n^{1/\bar{p}_2-1/\bar{p}_1}\sup_{0<\lambda\leq1}\lambda^{1-1/p}\|f_0\|_{L^{p}_{|x|}L_{\theta}^{\bar{p}_1}}\nonumber\\
		&=&w_n^{1/\bar{p}_2-1/\bar{p}_1}\|f_0\|_{L^{p}_{|x|}L_{\theta}^{\bar{p}_1}}.
	\end{eqnarray*}
	Therefore, we complete the proof.
\end{proof}
Obviously, Theorem \ref{T-4-1} improves the result of (\ref{formula1-6}).

\section*{Data availability statement}
Data sharing not applicable to this article as no datasets were generated or analysed during the current study.

\section*{Acknowledgements}
The authors would like to express their deep gratitude to the anonymous referees for their careful reading of the manuscript and their comments and suggestions. This work is supported by the National Natural Science Foundation of China (No. 11871452), the Natural Science Foundation of Henan Province (No. 202300410338) and the Nanhu Scholar Program for Young Scholars of Xinyang Normal University.\\


\end{document}